\documentclass{amsart}

\usepackage{amsmath,amsthm,amssymb,amsfonts}
\usepackage{bbm}
\usepackage[mathscr]{eucal}
\usepackage[english]{babel}

\usepackage{enumerate}

\allowdisplaybreaks[1]

\newtheorem{theorem}{Theorem}[section]

\newtheorem{proposition}[theorem]{Proposition}

\newtheorem{remark}[theorem]{Remark}

\numberwithin{equation}{section}

\def\CD{{\mathcal D}}
\def\CP{{\mathcal P}}
\def\CV{{\mathcal V}}
\def\BB{{\mathbb B}}

\def\RR{{\mathbb R}}
\def\TT{{\mathbb T}}

\begin{document}

\title[Gaussian bounds for heat kernels on the ball and simplex]
{Gaussian bounds for the heat kernels \\on the ball and simplex: Classical approach}

\author[G. Kerkyacharian]{Gerard Kerkyacharian}
\address{LPSM, CNRS-UMR 7599, and Crest}
\email{kerk@math.univ-paris-diderot.fr}

\author[P. Petrushev]{Pencho Petrushev}
\address{Department of Mathematics, University of South Carolina, Columbia, SC 29208}
\email{pencho@math.sc.edu}

\author[Y. Xu]{Yuan Xu}
\address{Department of Mathematics, University of Oregon, Eugene, Oregon 97403-1222}
\email{yuan@math.uoregon.edu}

\subjclass[2010]{42C05, 35K08}
\keywords{Heat kernel, Gaussian bounds, orthogonal polynomials, ball, simplex}
\thanks{The first author has been supported by ANR Forewer.
The second author has been supported by NSF Grant DMS-1714369.
The third author has been supported by NSF Grant DMS-1510296.}
\thanks{Corresponding author: Pencho Petrushev, E-mail: pencho@math.sc.edu}

\begin{abstract}
Two-sided Gaussian bounds are established for the weighted heat kernels on the unit ball and simplex in $\RR^d$
generated by classical differential operators whose eigenfunctions are algebraic polynomials.
\end{abstract}

\date{January 21, 2018}

\maketitle

\section{Introduction}

Two-sided Gaussian bounds have been established for heat kernels in various settings.
For example, Gaussian bounds for the Jacobi heat kernel on $[-1, 1]$
with weight $(1-x)^\alpha(1+x)^\beta$, $\alpha, \beta>-1$,
are obtained in \cite[Theorem 7.2]{CKP} and \cite[Theorem~5.1]{KPX},
and also in \cite{NS} in the case when $\alpha, \beta\ge -1/2$
(see \eqref{gauss-int} below).

In this article we establish two-sided Gaussian estimates for the heat kernels
generated by classical differential operators whose eigenfunctions are algebraic polynomials
in the weighted cases on the unit ball and simplex in $\RR^d$.
Such estimates are also established in \cite{KPX} using a general method that utilizes known
two-sided Gaussian estimates for the heat kernels generated by weighted Laplace operators on Riemannian manifolds.
Here we derive these results directly from the Gaussian bounds for the Jacobi heat kernel.
Such a direct method leads to working in somewhat restricted range
for the parameters of the weights (commonly used in the literature).
We next describe our results in detail.

We shall use standard notation. In particular, positive constants will be denoted by
$c, c', \tilde{c}, c_1, c_2, \dots$ and they may vary at every occurrence.
Most constants will depend on parameters that will be clear from the context.
The notation $a\sim b$ will stand for $c_1\le a/b \le c_2$.
The functions that we deal with in this article are assumed to be real-valued.

\subsection{Heat kernel on the unit ball}\label{subsec:ball}

Consider the operator
\begin{equation}\label{D-mu}
\CD_\mu:= \sum_{i=1}^d (1-x_i^2)\partial^2_i-2\sum_{1\le i < j \le d}x_i x_j\partial_i\partial_j
- (d+2 \mu)\sum_{i=1}^d x_i \partial_i,
\end{equation}
acting on sufficiently smooth functions on the unit ball
$\BB^d:=\big\{x\in\BB^d: \|x\|<1\big\}$ in $\RR^d$
equipped with the measure
\begin{equation}\label{def-meas-ball}
d\nu_\mu = w_\mu(x)dx := (1-\| x\|^2)^{\mu-1/2} dx, \quad \mu\ge 0,
\end{equation}
and the distance
\begin{equation}\label{dist-ball}
d_\BB(x,y) := \arccos \big(\langle x, y\rangle + \sqrt{1-\| x\|^2}\sqrt{1-\| y\|^2}\big),
\end{equation}
where $\langle x, y\rangle$ is the inner product of $x, y\in \RR^d$
and $\|x\|:= \sqrt{\langle x, x\rangle}$.
As will be shown the operator $\CD_\mu$ is symmetric and $-\CD_\mu$ is positive in $L^2(\BB, w_\mu)$.
Furthermore, $\CD_\mu$ is essentially self-adjoint.

Denote
\begin{equation}\label{def-ball-ball}
B_\BB(x, r):=\{y\in \RR^d: d_\BB(x,y)<r\}
\quad\hbox{and}\quad
V_\BB(x, r):= \nu_\mu(B_\BB(x, r)).
\end{equation}
As is well known (see, e.g. \cite[Lemma 11.3.6]{DaiX}) 
\begin{equation}\label{V-ball-ball}
V_\BB(x, r) \sim r^d(1-\|x\|^2+r^2)^\mu.
\end{equation}

Denote by $\CV_n(w_\mu)$ the set of all algebraic polynomials of degree $n$ in $d$ variables
that are orthogonal to lower degree polynomials in $L^2(\BB^d, w_\mu)$,
and let $\CV_0(w_\mu)$ be the set of all constants.
As is well known (see e.g. \cite[\S2.3.2]{DX}) $\CV_n(w_\mu)$, $n=0, 1, \dots$, are eigenspaces of the operator $\CD_\mu$,
more precisely,
\begin{equation}\label{eigen-ball}
\CD_\mu P=-n(n+d+2\mu-1)P,\quad \forall P\in\CV_n(w_\mu).
\end{equation}
Let $P_n(w_\mu; x, y)$ be the kernel of the orthogonal projector onto $\CV_n(w_\mu)$.
Then the semigroup $e^{t\CD_\mu}$, $t>0$, generated by $\CD_\mu$ has a (heat) kernel $e^{t\CD_\mu}(x,y)$ of the form
\begin{equation}\label{ball-HK}
e^{t\CD_\mu}(x,y)=\sum_{n=0}^\infty e^{-tn(n+2\lambda)}P_n(w_\mu; x, y),\quad \lambda:=\mu + (d-1)/2.
\end{equation}
We establish two-sided Gaussian bounds on $e^{t\CD_\mu}(x,y)$:

\begin{theorem}\label{thm:Gauss-ball}
For any $\mu\ge 0$ there exist constants $c_1, c_2, c_3, c_4>0$ such that for all $x, y\in \BB^d$ and $t>0$
\begin{equation} \label{gauss-ball}
\frac{c_1\exp\{- \frac{d_\BB(x,y)^2}{c_2t}\}}{\big[V_\BB(x, \sqrt t) V_\BB(y, \sqrt t)\big]^{1/2}}
\le  e^{tD_\mu}(x,y)
\le \frac{c_3\exp\{- \frac{d_\BB(x,y)^2}{c_4t}\}}{\big[V_\BB(x, \sqrt t) V_\BB(y, \sqrt t)\big]^{1/2}}.
\end{equation}
\end{theorem}

\subsection{Heat kernel on the simplex}\label{subsec:simplex}

We also establish two-sided Gaussian bounds for the heat kernel generated by the operator
\begin{equation}\label{def-D-simplex}
\CD_\kappa := \sum_{i=1}^d x_i\partial_i^2 - \sum_{i=1}^d\sum_{j=1}^d x_ix_j \partial_i\partial_j
+ \sum_{i=1}^d \big(\kappa_i + \tfrac12 - (|\kappa|+ \tfrac{n+1}{2}) x_i\big) \partial_i
\end{equation}
with $|\kappa|:=\kappa_1+\dots+\kappa_{d+1}$
acting on sufficiently smooth functions on the simplex
\begin{equation*}
\TT^d:=\Big\{x \in \RR^d: x_1 \ge 0,\dots, x_d\ge 0, |x| \le 1 \Big\},
\quad |x|:= x_1+\cdots+x_d,
\end{equation*}
in $\RR^d$, $d\ge 1$, equipped with the measure
\begin{equation}\label{def-meas-simpl}
d\nu_\kappa(x)=w_\kappa(x)dx:=\prod_{i=1}^d x_{i}^{\kappa_i-1/2}(1-|x|)^{\kappa_{d+1}-1/2}dx,
\quad  \kappa_i \ge 0,
\end{equation}
and the distance
\begin{equation}\label{def-dist-simpl}
d_\TT(x,y) := \arccos \Big(\sum_{i=1}^d \sqrt{x_i y_i} + \sqrt{1-|x|}\sqrt{1-|y|}\Big).
\end{equation}
As will be shown the operator $\CD_\kappa$ is symmetric and $-\CD_\kappa$ is positive
in the weighted space $L^2(\TT, w_\kappa)$,
furthermore, $\CD_\kappa$ is essentially self-adjoint.

We shall use the notation:
\begin{equation}\label{B-simplex}
B_\TT(x, r):=\{y\in\TT^d: \rho(x, y)<r\} \quad\hbox{and}\quad V_\TT(x, r):=\nu_\kappa(B(x, r)).
\end{equation}
It is known that
\begin{equation}\label{V-ball-simpl}
V_\TT(x, r) \sim r^d (1-|x|+r^2)^{\kappa_{d+1}}\prod_{i=1}^d(x_i+r^2)^{\kappa_i}.
\end{equation}
This equivalence follows e.g. from \cite[(5.1.10)]{DaiX},
see also (4.23)-(4.24) in \cite{KPX}.

Denote by $\CV_n(w_\kappa)$ the set of all algebraic polynomials of degree $n$ in $d$ variables
that are orthogonal to lower degree polynomials in $L^2(\TT^d, w_\kappa)$,
and let $\CV_0(w_\kappa)$ be the set of all constants.
As is well known (e.g. \cite[\S2.3.3]{DX}) $\CV_n(w_\kappa)$, $n=0, 1, \dots$,
are eigenspaces of the operator $\CD_\kappa$, namely,
\begin{equation}\label{sim-eigen-sp}
\CD_\kappa P=- n\big(n+|\kappa|+(d-1)/2\big)P,
\quad \forall P\in\CV_n(w_\kappa),\;\; n=0, 1, \dots.
\end{equation}
Let $P_n(w_\kappa; x, y)$ be the kernel of the orthogonal projector onto $\CV_n(w_\kappa)$ in $L^2(\TT^d, w_\kappa)$.
The heat kernel $e^{t\CD_\kappa}(x,y)$, $t>0$, takes the form
\begin{equation}\label{simplex-HK}
e^{t\CD_\kappa}(x,y)=\sum_{n=0}^\infty e^{-tn(n+\lambda_\kappa)} P_k(w_\kappa; x, y),
\quad \lambda_\kappa:= |\kappa|+(d-1)/2.
\end{equation}

\begin{theorem}\label{thm:Gauss-simpl}
For any $\kappa_i\ge 0$, $i=1, \dots, n+1$,
there are constants $c_1,c_2,c_3,c_4>0$ such that for all $x,y \in \TT^d$ and $t>0$
\begin{equation}\label{gauss-simplex}
\frac{c_1\exp\{- \frac{d_\TT(x,y)^2}{c_2t}\}}{\big[V_\TT(x, \sqrt t)V_\TT(y, \sqrt t)\big]^{1/2}}
\le  e^{t\CD_\kappa}(x,y)
\le \frac{c_3\exp\{- \frac{d_\TT(x,y)^2}{c_4t}\}}{\big[V_\TT(x, \sqrt t)V_\TT(y, \sqrt t)\big]^{1/2}}.
\end{equation}

\end{theorem}

\subsection{Method of proof and discussion}\label{subsec:method}

We shall prove Theorems~\ref{thm:Gauss-ball} and \ref{thm:Gauss-simpl}
by using the known two-sided Gaussian bounds on the Jacobi heat kernel on $[-1, 1]$.
We~next describe this result.
The classical Jacobi operator is defined by
\begin{equation}\label{def-Jacobi}
L_{\alpha, \beta}f(x):=\frac{\big[w_{\alpha,\beta}(x)(1-x^2)f'(x)\big]'}{w_{\alpha, \beta}(x)},
\end{equation}
where
\begin{equation*}
w_{\alpha, \beta}(x):=(1-x)^{\alpha}(1+x)^{\beta}, \quad \alpha, \beta>-1.
\end{equation*}
We consider $L_{\alpha, \beta}$ with domain $D(L):=\CP[-1, 1]$ the set of all algebraic polynomials restricted to $[-1, 1]$.
We also consider $[-1, 1]$ equipped with the weighted measure
\begin{equation}\label{mes-int}
d\nu_{\alpha, \beta}(x) := w_{\alpha, \beta}(x) dx = (1-x)^{\alpha}(1+x)^{\beta} dx
\end{equation}
and the distance
\begin{equation}\label{dist-dist}
\rho(x,y) := |\arccos x - \arccos y|.
\end{equation}
It is not hard to see that the Jacobi operator $L_{\alpha, \beta}$ in the setting described above
is essentially self-adjoint and $-L_{\alpha, \beta}$  is positive in $L^2([-1, 1], w_{\alpha, \beta})$.

We shall use the notation
\begin{equation}\label{int-ball}
B(x, r):=\{y\in [-1,1]: \rho(x, y)<r\} \quad\hbox{and}\quad V(x, r):=\nu_{\alpha, \beta}(B(x, r)).
\end{equation}
As is well known (see e.g. \cite[(7.1)]{CKP})
\begin{equation}\label{measure-ball}
V(x, r)\sim r(1-x+r^2)^{\alpha+1/2}(1+x+r^2)^{\beta+1/2}, \quad x\in [-1, 1], \; 0<r\le \pi.
\end{equation}

It is well known \cite{Sz} that the Jacobi polynomials $P_n^{(\alpha, \beta)}$, $n= 0, 1, \dots$,
are eigenfunctions of the operator $L_{\alpha, \beta}$, namely,
\begin{equation}\label{Jacobi-eigenv}
L_{\alpha, \beta}P_n^{(\alpha, \beta)}= -n(n+\alpha+\beta+1)P_n^{(\alpha, \beta)}, \quad n=0, 1, \dots.
\end{equation}
We consider the Jacobi polynomials $\big\{P_n^{(\alpha, \beta)}\big\}$ normalised in $L^2([-1,1], w_{\alpha, \beta})$.
Then the Jacobi heat kernel $e^{tL_{\alpha, \beta}}(x,y)$, $t>0$, takes the form
\begin{equation}\label{Jacobi-HK}
e^{tL_{\alpha, \beta}}(x,y)=\sum_{n=0}^\infty e^{-tn(n+\lambda)}P_n^{(\alpha, \beta)}(x)P_n^{(\alpha, \beta)}(y),
\quad \lambda:=\alpha+\beta+1.
\end{equation}

\begin{theorem}\label{thm:Gauss-int}
For any $\alpha, \beta>-1$
there exist constants $c_1, c_2, c_3, c_4 >0$ such that for all $x,y \in [-1,1]$ and $t>0$
\begin{equation}\label{gauss-int}
\frac{c_1\exp\{- \frac{\rho(x,y)^2}{c_2t}\}}{\big[V(x, \sqrt t)V(y, \sqrt t)\big]^{1/2}}
\le  e^{tL_{\alpha, \beta}}(x,y)
\le \frac{c_3\exp\{- \frac{\rho(x,y)^2}{c_4t}\}}{\big[V(x, \sqrt t)V(y, \sqrt t)\big]^{1/2}}.
\end{equation}
\end{theorem}

This theorem is established in \cite[Theorem~7.2]{CKP} using a general result on heat kernels
in Dirichlet spaces with a doubling measure and local Poincar\'{e} inequality.
The same theorem is also proved in \cite[Theorem~5.1]{KPX}.
In \cite{NS} Nowak and Sj\"{o}gren obtained this result in the case when $\alpha, \beta \ge -1/2$
via a direct method using special functions.

For the proof of Theorem~\ref{thm:Gauss-ball} it will be critical that the kernel $P(w_\mu; x, y)$
of the orthogonal projector onto $\CV_n(w_\mu)$ in $L^2(\BB, w_\mu)$ has an explicit representation
in terms of the univariate Gegenbauer polynomials (see \eqref{rep-Pn-ball}-\eqref{rep-Pn-ball-2}).
For the proof of Theorem~\ref{thm:Gauss-simpl} we deploy the well known representation
of the kernel $P_n(w_\kappa; x, y)$ in terms of Jacobi polynomials (see \eqref{rep-Pn-simpl}).

It should be pointed out that our method of proof of estimates \eqref{gauss-ball} and \eqref{gauss-simplex}
works only in the range $\mu \ge 0$ for the weight parameter in the case of the ball
and in the range $\kappa_i\ge 0$, $i=1,\dots, n$, in the case of the simplex.
These restrictions on the range of the parameters are determined by the range
for the parameters in the representations of the kernels $P(w_\mu; x, y)$ and $P_n(w_\kappa; x, y)$.

Observe that the two-sided estimates on the heat kernels from \eqref{gauss-ball},\eqref{gauss-simplex}
coupled with the general results from \cite{CKP, KP}
entail smooth functional calculus in the settings on the ball and simplex (see \cite{IPX,PX2}),
in particular, the finite speed propagation property is valid.
For more details, see \cite[\S 3.1]{KPX}.

\section{Proof of Gaussian bounds for the heat kernel on the ball}
\label{sec:proof-ball}

We adhere to the notation from \S \ref{subsec:ball}.
Define
\begin{equation*}
D_{i,j}:=x_i\partial_j-x_j\partial_i,\quad 1 \le i \ne j \le d.
\end{equation*}
It is easy to see that
\begin{equation} \label{Dij-theta}
D_{i,j} = \partial_{\theta_{i,j}}
\quad \hbox{with} \quad
(x_i,x_j) = r_{i,j}(\cos \theta_{i,j},\sin \theta_{i,j}).
\end{equation}
Further, define the second order differential operators
$$
D^2_{i,i} := [w_\mu(x)]^{-1} \partial_i \left[(1-\|x\|^2) w_\mu(x) \right]\partial_i,
\quad 1 \le i \le d.
$$
It turns out that the differential operator $D_\mu$ from \eqref{D-mu} can be decomposed
as a~sum of second order differential operators \cite[Proposition 7.1]{DaiX}:
\begin{equation}\label{decomp}
  \CD_\mu = \sum_{i=1}^d D^2_{i,i} +  \sum_{1\le i < j \le d} D^2_{i,j}
      =  \sum_{1\le i \le j \le d} D^2_{i,j}.
\end{equation}

The basic properties of the operator $D_\mu$ are given in the following

\begin{theorem}\label{thm:prop-D}
For $f \in C^2(\BB^d)$ and $g \in C^1(\BB^d)$,
\begin{align}\label{rep-Dmu}
& \int_{\BB^d} \CD_\mu f(x) g(x) w_\mu(x) dx
\\
&= - \int_{\BB^d}
\Big[ \sum_{i =1}^d  \partial_i f (x)\partial_i g(x) (1-\|x\|)^2 +
       \sum_{1 \le i< j \le d} D_{i,j} f(x)D_{i,j} g(x) \Big] w_\mu(x) dx. \notag
\end{align}
Consequently, the operator $D_\mu$ is essentially self-adjoint
and $-D_\mu$ is positive in $L^2(\BB^d, w_\mu)$.
\end{theorem}

\begin{proof}
Applying integration by parts in the variable $x_i$ we obtain
\begin{align*}
  \int_{\BB^d} (D_{i,i}^2 f(x)) g(x) w_\mu(x) dx
   & = \int_{\BB^d}  \left( \partial_i  \left[ (1-\|x\|^2) w_\mu(x)
        \partial_i  f(x)  \right] \right) g(x) dx  \\
   & = - \int_{\BB^d}   \partial_i f (x)\partial_i g(x) (1-\|x\|)^2 w_\mu(x) dx.
\end{align*}
We now handle $D_{i,j}^2$. It is sufficient to consider $D_{1,2}$.
If $d =2$ we switch to polar coordinates and use \eqref{Dij-theta} and
integration by parts for $2\pi$-periodic functions to obtain
\begin{align*}
  \int_{\BB^2} (D_{i,j}^2 f(x)) g(x) w_\mu(x) dx
  & = \int_0^1 r (1-r^2)^{\mu-1} \int_{0}^{2\pi}( \partial_\theta^2 f) g d\theta dr \\
  & = - \int_0^1 r (1-r^2)^{\mu-1} \int_{0}^{2\pi}  \partial_\theta f \partial_\theta g d\theta dr \\
   & = - \int_{\BB^2} D_{i,j} f(x) D_{i,j} g(x) w_\mu(x) dx.
\end{align*}
In dimension $d > 2$ we apply the following integration identity that follows by
a~simple change of variables,
\begin{equation}\label{B-B}
\int_{\BB^d} f(x)dx  = \int_{\BB^{d-2}}
\Big[\int_{\BB^2} f\big(\sqrt{1-\|v\|^2}u, v\big) du \Big] (1-\|v\|^2) dv,
\end{equation}
and parametrizing the integral over $\BB^2$ by polar coordinates we arrive at
\begin{equation*}
\int_{\BB^d} (D_{i,j}^2 f(x)) g(x) w_\mu(x) dx = -\int_{\BB^d} D_{i,j} f(x) D_{i,j} g(x) w_\mu(x) dx.
\end{equation*}
The above identities imply \eqref{rep-Dmu}.

\smallskip

We consider the operator $\CD_\mu$ with domain $D(\CD_\mu)= \CP(\BB^d)$ the set of all polynomial on $\BB^d$,
which is obviously dense in $L^2(\BB^d, w_\mu)$.
From \eqref{rep-Dmu} it readily follows that the operator $\CD_\mu$ is symmetric and $-\CD_\mu$ is positive.

We next show that the operator $\CD_\mu$ is essentially self-adjoint,
that is, the completion $\overline{\CD}_\mu$ of the operator $\CD_\mu$ is self-adjoint.
Let $\{P_{nj}: j=1, \dots, \dim \CV_n\}$ be an orthonormal basis of $\CV_n=\CV_n(w_\mu)$ consisting of real-valued polynomials.
Clearly
\begin{equation*}
D(\CD_\mu)=\Big\{f= \sum_{n, j} a_{nj} P_{nj}: \; a_{nj}\in\RR,
\;\;\{a_{nj}\}\;\;\hbox{compactly supported}\Big\}, \;\;\hbox{and}
\end{equation*}
\begin{equation*}
\CD_\mu f= -\sum_{n, j} a_{nj} n(n+2\lambda) P_{nj}
\quad \hbox{if}\quad f= \sum_{j} a_{nj} P_{nj}\in D(\CD_\mu).
\end{equation*}
We define $\overline{\CD}_\mu$ and its domain $D(\overline{\CD}_\mu)$ by
\begin{equation*}
D(\overline{\CD}_\mu):=\Big\{f= \sum_{n=0}^\infty\sum_{j=1}^{\dim \CV_n} a_{nj} P_{nj}:
\; \sum_{n,j} |a_{nj}|^2 <\infty,\;\;\sum_{n,j} |a_{nj}|^2 (n(n+2\lambda))^2<\infty \Big\}
\end{equation*}
and
\begin{equation*}
\overline{\CD}_\mu f:= -\sum_{n,j} a_{nj} n(n+2\lambda) P_{nj}
\quad \hbox{if}\quad f= \sum_{n, j} a_{nj} P_{nj}\in D(\overline{\CD}_\mu).
\end{equation*}
It is easily to show that $\overline{\CD}_\mu$ is the closure of $\CD_\mu$
and that $\overline{\CD}_\mu$ is self-adjoint.
\end{proof}

\begin{remark}\label{rem:green-ball}
Identity \eqref{rep-Dmu} is the weighted Green's formula on $\BB^d$ $($see \cite{KPX}$)$.
\end{remark}

\begin{proof}[Proof of Theorem~\ref{thm:Gauss-ball}]
We shall assume that $0<t\le 1$. In the case $t>1$ the Gaussian bounds \eqref{gauss-ball}
obviously follow from \eqref{gauss-ball} in the case $t=1$.

It is known (see \cite[Thm. 5.2.8]{DX}) that for $\mu > 0$ the kernel $P_n(w_\mu;x,y)$ of
the orthogonal projector onto $\CV_n(w_\mu)$ in $L^2(\BB, w_\mu)$
has the representation
\begin{equation}\label{rep-Pn-ball}
P_n(w_\mu;x,y) = c_\lambda\frac{n+\lambda}{\lambda}
\int_{-1}^1 C_n^\lambda \left(\langle x, y\rangle + u \sqrt{1-\|x\|^2}\sqrt{1-\|y\|^2}\right)
       (1-u^2)^{\mu-1} du,
\end{equation}
where $C_n^\lambda$ is the Gegenbauer polynomial of degree $n$
and $c_\lambda>0$ is a constant depending only on $\lambda$ and $d$.
The Gegenbauer polynomials $\{C_n^\lambda\}$ are orthogonal in the weighted space $L^2([-1, 1], w_\lambda)$
with $w_\lambda(u) := (1-u^2)^{\lambda-1/2}$
and can be defined by the generating function
\begin{equation*}
(1-2uz-z^2)^{-\mu}=\sum_{n=0}^\infty C_n^\lambda(u)z^n, \quad |z|<1, \; |u|<1.
\end{equation*}
Using that $C_n^\lambda(1)=\binom{n+2\lambda-1}{n}$ it is easy to show that
\begin{equation}\label{geg}
\int_{-1}^{1}|C_n^\lambda(u)|^2w_\lambda(u)du=\frac{\lambda}{n+\lambda}C_n^\lambda(1).
\end{equation}
In the limiting case $\mu = 0$ the representation of $P_n(w_\mu;x,y)$ takes the form
\begin{align}\label{rep-Pn-ball-2}
 P_n(w_0; x,y) = c_d\frac{\lambda+n}{\lambda} & \left[ C_n^\lambda
   \left(\langle x,y\rangle +\sqrt{1-\|x\|^2} \sqrt{1-\|y\|^2}\right) \right .\\
& \left . + C_n^\lambda \left (\langle x,y\rangle - \sqrt{1-\|x\|^2} \sqrt{1-\|y\|^2} \right)\right]. \notag
\end{align}

If $\alpha=\beta=\lambda-1/2$ we denote the Jacobi operator by
$L_\lambda:=L_{\lambda-1/2, \lambda-1/2}$
and we have
$L_\lambda f(x) = (1-x^2) f''(x) - (2\lambda+1) f'(x)$.
We denote by  $e^{tL_\lambda}(u,v)$ the Jacobi heat kernel in this case
and by \eqref{Jacobi-HK} and \eqref{geg} we obtain
\begin{equation}\label{gegen-HK}
e^{tL_\lambda}(u,v)
= \sum_{n =0}^\infty e^{-tn(n+2\lambda)}\frac{n+\lambda}{\lambda}\frac{C_n^\lambda(u) C_n^\lambda(v)}{C_n^\lambda(1)},
\quad \lambda:=\mu+(d-1)/2.
\end{equation}

Assume $\mu>0$.
The above, \eqref{ball-HK}, and \eqref{rep-Pn-ball} lead to the representation
\begin{equation} \label{h-kernel-ball2}
e^{t \CD_\mu}(x,y) = c_\lambda\int_{-1}^1 e^{t L_\lambda} \left(1,\langle x, y\rangle
+ u \sqrt{1-\|x\|^2}\sqrt{1-\|y\|^2}\right)(1-u^2)^{\mu-1} du.
\end{equation}

Note that in the case of Gegenbauer polynomials ($\alpha = \beta = \lambda -1/2$)
by \eqref{measure-ball} it follows that
$V(x,r) \sim r( 1-x^2 + r^2)^\lambda$, $-1 \le x \le 1$,
and hence
\begin{equation}\label{V-ball}
V(1,\sqrt{t})\sim t^{\lambda+1/2}
\quad\hbox{and}\quad
V(z,\sqrt{t}) \sim t^{\lambda+1/2}(1+(1-z^2)/t)^\lambda,
\quad |z|\le 1.
\end{equation}
If $x = \cos \theta$, then $1-x =2 \sin^2 \frac{\theta}{2} \sim \theta^2$
and hence
\begin{equation}\label{dist-1z}
\rho(1,z) = |\arccos 1 - \arccos z| = \arccos z \sim \sqrt{1-z}, \quad -1\le z\le 1.
\end{equation}
From this, \eqref{h-kernel-ball2}, \eqref{gauss-int}, and \eqref{V-ball} we obtain
\begin{equation} \label{main-est}
e^{t \CD_\mu}(x,y)
\le c_1\int_{-1}^1 \frac{\exp \big\{-\frac{1-z(u;x,y)}{c_2t}\big\}}
{t^{\lambda+1/2}\big(1+  \frac{1-z(u;x,y)^2}{t}\big)^\lambda}(1-u^2)^{\mu-1} du
\end{equation}
and
\begin{equation} \label{main-est1}
e^{t \CD_\mu}(x,y)
\ge c_3\int_{-1}^1 \frac{\exp\big\{-\frac{1-z(u;x,y)}{c_4t}\big\}}
{t^{\lambda+1/2}\big(1+  \frac{1-z(u;x,y)^2}{t}\big)^\lambda}(1-u^2)^{\mu-1} du,
\end{equation}
where $z(u; x,y) := \langle x, y\rangle + u\sqrt{1-\|x\|^2}\sqrt{1-\|y\|^2}$.

Since $1 + b \le e^{b}$ for $b \ge 0$, we have
$$
1 \le \Big(1+  \frac{1-z^2}{t}\Big)^\lambda \le  \Big(1+ 2\frac{1-z}{t}\Big)^\lambda
\le \exp\Big\{2\lambda \frac{1-z}{t}\Big\},
\quad |z|\le 1.
$$
Therefore, by replacing the constant $c_4$ in \eqref{main-est1} by a smaller constant $c_4'$ we get
\begin{equation} \label{main-est2}
e^{t \CD_\mu}(x,y)
\ge \frac{c_3'}{t^{\lambda+1/2}}\int_{-1}^1 \exp\Big\{-\frac{1-z(u;x,y)}{c_4't}\Big\}
(1-u^2)^{\mu-1} du.
\end{equation}
Obviously, from \eqref{main-est} it follows that
\begin{equation} \label{main-est3}
e^{t \CD_\mu}(x,y)
\le \frac{c_1}{t^{\lambda+1/2}}\int_{-1}^1\exp\Big\{-\frac{1-z(u;x,y)}{c_2t}\Big\}
(1-u^2)^{\mu-1} du.
\end{equation}

We have
\begin{align*}
1-z(u;x,y) =  1- \langle x, y\rangle - \sqrt{1-\|x\|^2}\sqrt{1-\|y\|^2}
    + (1-u)\sqrt{1-\|x\|^2}\sqrt{1-\|y\|^2}
\end{align*}
and using the definition of $d_\BB(x,y)$ in \eqref{dist-ball} we get
$$
 1- z(1;x,y) = 1- \cos d_\BB(x,y) = 2 \sin^2 \frac{d_\BB(x,y)}{2} \sim d_\BB(x,y)^2.
$$
Hence,
$$
  1-z(u;x,y) \sim d_\BB(x,y)^2 + (1-u) H(x,y), \quad H(x,y): =    \sqrt{1-\|x\|^2}\sqrt{1-\|y\|^2}.
$$
Consequently,
$$
\exp\Big\{-\frac{1-z(u;x,y)}{c_2t}\Big\}
\le \exp\Big\{-\frac{d_\BB(x,y)^2}{c't}\Big\}
\exp\Big\{-\frac{(1-u)H(x,y)}{c't}\Big\}
$$
and
$$
\exp\Big\{-\frac{1-z(u;x,y)}{2c_4t}\Big\}
\ge \exp\Big\{-\frac{d_\BB(x,y)^2}{c''t}\Big\}
\exp\Big\{-\frac{(1-u)H(x,y)}{c''t}\Big\}.
$$
These two inequalities along with \eqref{main-est2}-\eqref{main-est3}
imply that in order to obtain the two-sided Gaussian bounds in \eqref{gauss-ball}
it suffices to show that the quantity
\begin{equation}\label{def-At}
A_t(x,y): =\frac{1}{t^{\lambda+1/2}} \int_{-1}^1\exp\Big\{-\frac{(1-u)H(x,y)}{ct}\Big\}
(1-u^2)^{\mu-1}du
\end{equation}
satisfies the following inequalities, for any $\varepsilon>0$,
\begin{align}\label{key-est}
\frac{c^\star}{\big[V_\BB(x,\sqrt{t})V_\BB(y,\sqrt{t})\big]^{1/2}}
\le A_t(x,y)
\le \frac{c^{\star\star}\exp\Big\{\varepsilon\frac{d_\BB(x,y)^2}{t}\Big\}}{\big[V_\BB(x,\sqrt{t})V_\BB(y,\sqrt{t})\big]^{1/2}}.
\end{align}
Here the constant $c^{\star\star}>0$ depends on $\varepsilon$.

\medskip

\noindent
{\em Lower bound estimate.}
First, assume that $H(x,y)/ t \ge 1$. Then we have
\begin{align}\label{lower-est}
A_t(x,y)
& \ge \frac{\tilde{c}}{t^{\lambda+1/2}}\int_0^1\exp\Big\{-\frac{(1-u)H(x,y)}{ct}\Big\}(1-u)^{\mu-1} du \notag
\\
& =\frac{\tilde{c}t^\mu}{t^{\lambda+1/2}H(x,y)^\mu}\int_0^{H(x,y)/t}v^{\mu-1}e^{-v/c}dv
\\
& \ge \frac{c_* }{t^{d/2}H(x,y)^\mu} \quad \hbox{with}
\quad  c_*=\tilde{c}\int_0^1 v^{\mu-1} e^{-v/c} dv, \notag
\end{align}
where we applied the substitution $v=(1-u)H(x, y)/t$ and used that $\lambda +1/2 = \mu + d/2$.
However,
by \eqref{V-ball-ball},  $V_\BB(x,r) \ge cr^d(1-\|x\|^2)^\mu$, which implies
$$
t^{d/2} H(x,y)^\mu = t^{d/2} \big(\sqrt{1-\|x\|^2} \sqrt{1-\|y\|^2}\big)^\mu
\le \big[V_\BB(x,\sqrt{t})V_\BB(y,\sqrt{t}\big]^{1/2}.
$$
Putting the above together we conclude that $A_t(x,y)$ obeys the lower bound in \eqref{key-est} in this case.

Now, assume that $H(x,y)/t \le 1$.
Then $\exp \big\{-\frac{(1-u)H(x,y)}{ct}\big\} \ge e^{-1/c}$
and we have
$$
A_t(x,y)\ge \frac{\tilde{c}}{t^{\lambda+1/2}} \ge \frac{c^\star}{\left[V_\BB(x,\sqrt{t})V_\BB(y,\sqrt{t})\right]^{1/2}}.
$$
Here we used that, by \eqref{V-ball-ball}, $V_\BB(x,r) \ge cr^{d+2\mu} = cr^{2\lambda+1}$.
Thus, $A_t(x,y)$ again obeys the lower bound estimate in \eqref{key-est}
and this completes its proof.

\medskip

\noindent
{\em Upper bound estimate.}
Obviously $\exp\big\{-\frac{(1-u)H(x,y)}{ct}\big\} \le 1$
and hence
\begin{equation}\label{At-upper1}
A_t(x,y) \le \frac{c_*}{t^{\lambda+1/2}}= \frac{c_*}{t^{d/2+\mu}}. 
\end{equation}

We shall obtain another estimate on $A_t(x,y)$ by breaking the integral in \eqref{def-At} 
into two parts: one over $[0,1]$ and the other over $[-1,0]$.
Just as in \eqref{lower-est} applying the substitotion $v=(1-u)H(x,y)/t$ we obtain
\begin{align*}
\frac{1}{t^{\lambda+1/2}} \int_0^1\exp\Big\{-\frac{(1-u)H(x,y)}{ct}\Big\}(1-u^2)^{\mu-1}du
\le \frac{c^*\max\{1, 2^{\mu-1}\}}{t^{d/2}H(x,y)^\mu}
\end{align*}
with
$c^*=\int_0^\infty v^{\mu-1} e^{-v/c} dv$.
Here we used that $(1+u)^{\mu-1}\le \max\{1, 2^{\mu-1}\}$.

For the integral over $[-1,0]$ we use the fact that $1-u \ge 1$ for $u \in [-1,0]$
to obtain
\begin{align*}
\frac{1}{t^{\lambda+1/2}} \int_{-1}^0\exp\Big\{-\frac{(1-u)H(x,y)}{ct}\Big\}(1-u^2)^{\mu-1}du
&\le \frac{c_*}{t^{\lambda+1/2}}\exp\Big\{-\frac{H(x,y)}{ct}\Big\}
\\
\le \frac{\tilde{c}}{t^{\lambda+1/2}}\Big(\frac{t}{H(x, y)}\Big)^\mu
&= \frac{\tilde{c}}{t^{\lambda+1/2}H(x, y)^\mu}.
\end{align*}
Here we used that $v^\mu \le \lfloor \mu+1\rfloor! e^v$, $\forall v>0$,
and $\lambda=\mu+(d-1)/2$.

Together, the above inequalities imply
\begin{equation}\label{At-upper2}
    A_t(x,y) \le \frac{c^*}{t^{d/2} H(x,y)^\mu}.
\end{equation}
In turn, \eqref{At-upper1} and \eqref{At-upper2} yield
\begin{equation}\label{At-upper}
A_t(x,y) \le \frac{c_\diamond}{t^{d/2} (t + H(x,y))^\mu}.
\end{equation}

It remains to show that the above estimate implies the upper bound estimate in \eqref{key-est}.
To this end we need the following simple inequalities:
\begin{equation} \label{elementary}
(u+a) (u+b) \le 3 (u^2 +a b) (1+u^{-1} |a-b|), \quad a, b \ge 0, \;\; 0 < u \le 1,
\end{equation}
(see, e.g. \cite[(2.21]{PX1}) and (see \cite[(4.9)]{PX2})
$$
\big|\sqrt{1-\|x\|^2}- \sqrt{1-\|y\|^2}\big| \le \sqrt{2} d_\BB(x,y), \quad x,y \in \BB^d.
$$
Together, these two inequalities yield
\begin{equation}\label{combo}
\big(\sqrt{t} + \sqrt{1-\|x\|^2}\big)\big(\sqrt{t} + \sqrt{1-\|y\|^2}\big)
\le c \big(t + H(x,y)\big)\Big(1+ \frac{d_\BB(x,y)}{\sqrt{t}}\Big).
\end{equation}
Evidently $1+u\le \varepsilon^{-1}e^{\varepsilon u}$ for $u\ge 0$ and $0<\varepsilon\le 1$,
and hence
\begin{equation}\label{ineq-3}
(1+b)^\mu \le 2^\mu(1+b^2)^{\mu/2} \le 2^\mu \varepsilon^{-\mu/2} e^{\mu\varepsilon b^2},
\quad \forall b\ge 0, \; 0<\varepsilon\le 1.
\end{equation}
Also, from \eqref{V-ball-ball} it follows that
$V_\BB(x, r)\sim r^d\big(r+\sqrt{1-\|x\|^2}\big)^{2\mu}$.
From this, \eqref{combo}, and \eqref{ineq-3} it follows that
\begin{align*}
\big[V_\BB(x,\sqrt{t})V_\BB(y,\sqrt{t})\big]^{1/2}
& \le c t^{d/2} \big(\sqrt{t} + \sqrt{1-\|x\|^2}\big)^\mu \big(\sqrt{t} + \sqrt{1-\|y\|^2}\big)^\mu
\\
& \le c_\varepsilon t^{d/2}\big(t + H(x,y)\big)^\mu \exp\Big\{\mu\varepsilon \frac{d_\BB(x,y)^2}{t}\Big\}.
\end{align*}
In turn, this and \eqref{At-upper} yield the upper bound estimate in \eqref{key-est}.

\smallskip

We next consider the case when $\mu=0$.
Now, \eqref{ball-HK}, \eqref{rep-Pn-ball-2}, and \eqref{gegen-HK} yield the representation
\begin{align}\label{rep-D0}
e^{t \CD_0}(x,y)
&= c_\lambda e^{t L_\lambda} \left(1,\langle x, y\rangle + \sqrt{1-\|x\|^2}\sqrt{1-\|y\|^2}\right)
\\
&+ c_\lambda e^{t L_\lambda} \left(1,\langle x, y\rangle - \sqrt{1-\|x\|^2}\sqrt{1-\|y\|^2}\right).\notag
\end{align}
From this point on the proof follows in the footsteps of the proof when $\mu>0$ from above,
but is much simpler because the integral in \eqref{h-kernel-ball2} is replaced in \eqref{rep-D0} by two terms.
We omit the further details.
The proof of Theorem~\ref{thm:Gauss-ball} is complete.
\end{proof}

\section{Proof of Gaussian bounds for the heat kernel on the simplex}\label{sec:simplex}

In this part we adhere to the notation from \S\ref{subsec:simplex}.
The differential operator $\CD_\kappa$ from \eqref{def-D-simplex} can be written in the more symmetric form
\begin{equation}\label{rep-D-simplex}
\CD_\kappa = \sum_{i=1}^d U_{i} + \sum_{1\le i < j \le d} U_{i,j},
\end{equation}
where, with the notation $\partial_{i,j} : = \partial_i - \partial_j$,
\begin{align*}
  U_{i} & := \frac{1}{w_k(x)} \partial_i (x_i(1-|x|) w_\kappa(x)) \partial_i,\\
  U_{i,j} & := \frac{1}{w_k(x)} \partial_{i,j} (x_i x_j w_\kappa(x)) \partial_{i,j}, \quad 1 \le i \le d.
\end{align*}
This decomposition was first established in \cite{BSX} for $w_\kappa(x) = 1$ and later used in \cite{Diz} for $w_\kappa$.
It is easy to verify it directly.
The following basic property of the operator $\CD_\kappa$ follows immediately from \eqref{rep-D-simplex} by integration by parts:

\begin{proposition}
For any $f \in C^2(\TT^d)$ and $g \in C^1(\TT^d)$,
\begin{align}\label{diver-simplex}
\int_{\TT^d} \CD_\kappa f(x) \cdot g(x) w_\kappa(x) dx
=&-\int_{\TT^d}\Big[ \sum_{i =1}^d \partial_i f (x) \partial_i g (x) x_i (1-|x|)
\\
&+ \sum_{1 \le i \le j \le d} \partial_{i,j} f (x) \partial_{i,j} g (x) x_i x_j \Big] w_\kappa(x) dx.\notag
\end{align}
\end{proposition}

Observe that identity \eqref{diver-simplex} is the weighted Green's formula on the simplex $\TT^d$ (see \cite{KPX}).

We consider the operator $\CD_\kappa$ defined on the set $D(\CD_\kappa)=\CP(\TT^d)$ of all algebraic
polynomials on $\TT^d$, which is obviously dense in $L^2(\TT^d, w_\kappa)$.
From \eqref{diver-simplex} it readily follows that the operator $\CD_\kappa$ is
symmetric and $-\CD_\kappa$ is positive in $L^2(\TT^d, w_\kappa)$.
Furthermore, just as in the proof of Theorem~\ref{thm:prop-D} it follows that the operator $\CD_\kappa$ is essentially self-adjoint.

\smallskip

\begin{proof}[Proof of Theorem~\ref{thm:Gauss-simpl}]
We may assume that $0<t\le 1$, because the case $t>1$ follow immediately from the case $t=1$.

Recall that we consider in this article the Jacobi polynomials $P_n^{(\alpha, \beta)}$, $n=0,1, \dots$,
normalized in $L^2([-1,1], w_{\alpha, \beta})$.
It is known (see \cite[Theorem 5.3.4]{DX}) that if all $\kappa_i > 0$ the kernel $P_n(w_\mu;x,y)$
of the orthogonal projector onto $\CV_n(w_\kappa)$ in $L^2(\TT^d, w_\kappa)$ has the following representation
\begin{align}\label{rep-Pn-simpl}
P_n(w_\mu;x,y) = & c_\kappa P_n^{(\lambda -\frac 12,-\frac 12)}(1) \notag
\\
&\times \int_{[-1,1]^{d+1}}P_n^{(\lambda -\frac 12,-\frac 12)} \left(2z(u;x,y)^2-1\right)
\prod_{i=1}^{d+1}  (1-u_i^2)^{\kappa_i} du,
\end{align}
where
\begin{equation*} 
z(u;x,y) := \sum_{k=1}^{d+1} u_i\sqrt{x_iy_i}, \;\; x_{d+1} := 1-|x|, \;\; y_{d+1} := 1-|y|, \;\; \lambda:=|\kappa|+(d-1)/2,
\end{equation*}
in which $|x| = x_1+ \ldots + x_d$. In the case when some or all $\kappa_i =0$, this identity holds under the
limit $\kappa_i \to 0$, using that
$$
\lim_{\kappa \to 0+}   \frac{\int_{-1}^1 f(x) (1-x^2)^{\kappa - 1} dx }{\int_{-1}^1 (1-x^2)^{\kappa - 1} dx}
= \frac12  \left[f(1) + f(-1) \right].
$$

Assume $\kappa_i>0$, $i=1, \dots, n+1$. Combining \eqref{simplex-HK}, \eqref{Jacobi-HK}, and
\eqref{rep-Pn-simpl} we obtain the representation
\begin{equation} \label{h-kernel-simplex2}
e^{t\CD_\kappa}(x,y) = c_\kappa\int_{[-1,1]^{d+1}}
e^{tL_{\lambda -\frac 12,-\frac 12}} \left(1, 2z(u;x,y)^2-1\right)\prod_{i=1}^{d+1}  (1-u_i^2)^{\kappa_i} du.
\end{equation}
Note that from \eqref{def-dist-simpl} we have
$\sum_{k=1}^{d+1} \sqrt{x_iy_i}=\cos d_\TT(x, y)$
and hence $|z(u;x,y)|\le 1$.
Just as in \eqref{dist-1z} we obtain
\begin{equation}\label{rho-1z}
\rho(1, 2z^2 -1):= |\arccos 1 - \arccos (2z^2 -1)|
 \sim \sqrt{1 - (2z^2 -1)} \sim \sqrt{1-z^2}.
\end{equation}
On the other hand, with $\alpha= \lambda-1/2$ and $\beta=-1/2$ we infer from \eqref{measure-ball} that
$ V(x, r)\sim r(1-x+r^2)^\lambda$ and hence
\begin{equation*}
V(1, \sqrt{t})\sim t^{\lambda+1/2}
\;\;\hbox{and}\;\;
V(2z^2-1, \sqrt{t})\sim t^{1/2}\big(t + 2(1-z^2)\big)^{\lambda}
\sim t^{\lambda+1/2}\Big(1 + \frac{1-z^2}{t}\Big)^{\lambda}.
\end{equation*}
We use these equivalences, \eqref{h-kernel-simplex2}, \eqref{gauss-int}, and \eqref{rho-1z} to obtain
\begin{equation} \label{main-est-TT}
e^{t\CD_\kappa}(x,y)
\le c_1\int_{[-1,1]^{d+1}}\frac{\exp \big\{-\frac{1-z(u;x,y)^2}{c_2t}\big\}}
{t^{\lambda+1/2}\big(1+\frac{1-z(u;x,y)^2}{t}\big)^\lambda} \prod_{i=1}^{d+1} (1-u_i^2)^{\kappa_i-1} du
\end{equation}
and
\begin{equation} \label{main-est1-TT}
e^{t\CD_\kappa}(x,y)
\ge c_3\int_{[-1,1]^{d+1}}\frac{\exp\big\{-\frac{1-z(u;x,y)^2}{c_4t}\big\}}
{t^{\lambda+1/2}\big(1+\frac{1-z(u;x,y)^2}{t}\big)^\lambda} \prod_{i=1}^{d+1} (1-u_i^2)^{\kappa_i-1} du.
\end{equation}
Just as in the proof of Theorem~\ref{thm:Gauss-ball} by replacing
the constant $c_4$ in \eqref{main-est1-TT} by a smaller constant $c_4'$ we can eliminate the term
$\big(1+  \frac{1-z(u;x,y)^2}{t}\big)^\lambda$ in the denominator.
Thus, it follows that
\begin{equation} \label{main-est2-TT}
e^{t\CD_\kappa}(x,y)
\ge \frac{c_3'}{t^{\lambda+1/2}}
\int_{[-1,1]^{d+1}}\exp\Big\{-\frac{1-z(u;x,y)^2}{c_4't}\Big\}\prod_{i=1}^{d+1} (1-u_i^2)^{\kappa_i-1}du.
\end{equation}
By simply deleting that term in \eqref{main-est-TT} we get
\begin{equation*} 
e^{t\CD_\kappa}(x,y)
\le \frac{c_1}{t^{\lambda+1/2}}
\int_{[-1,1]^{d+1}}\exp\Big\{-\frac{1-z(u;x,y)^2}{c_2t}\Big\}\prod_{i=1}^{d+1} (1-u_i^2)^{\kappa_i-1}du.
\end{equation*}
Evidently,
$$
1 - z(u;x,y)^2 = (1+|z(u;x,y)|)(1 - |z(u;x,y)|) \ge 1-|z(u,x,y)| \ge 1- \sum_{i=1}^{d+1}|u_i| \sqrt{x_iy_i}.
$$
Using the symmetry of the last term above with respect to sign changes of $u_i$,
and that $1-u_i^2\sim 1-u_i$ when $0\le u_i\le 1$,
we conclude that
\begin{equation} \label{main-est3-TT}
e^{t\CD_\kappa}(x,y)
\le \frac{c_1'}{t^{\lambda+1/2}}
\int_{[0,1]^{d+1}}\exp\Big\{-\frac{1-z(u;x,y)}{c_2t}\Big\}\prod_{i=1}^{d+1} (1-u_i)^{\kappa_i-1}du.
\end{equation}
Similarly, using that $1-z(u; x,y)^2\le 2(1-z(u; x,y))$
we infer from \eqref{main-est2-TT} that
\begin{equation} \label{main-est4-TT}
e^{t\CD_\kappa}(x,y)
\ge \frac{c_3''}{t^{\lambda+1/2}}
\int_{[0,1]^{d+1}}\exp\Big\{-\frac{1-z(u;x,y)}{c_4''t}\Big\}\prod_{i=1}^{d+1} (1-u_i)^{\kappa_i-1}du.
\end{equation}

By the definition of $d_\TT(x,y)$ in \eqref{def-dist-simpl} we have
$$
1-  \sum_{i=1}^{d+1} \sqrt{x_iy_i} = 1-\cos d_\TT(x, y) = 2\sin^2\frac {d_\TT(x,y)}{2} \sim d_\TT(x,y)^2
$$
and hence
\begin{equation}\label{z-simplex}
1-z(u;x,y) = 1-  \sum_{i=1}^{d+1} \sqrt{x_iy_i} + \sum_{i=1}^{d+1} (1-u_i) \sqrt{x_iy_i}
 \sim d_\TT(x,y)^2 +  \sum_{i=1}^{d+1} (1 - u_i) \sqrt{x_iy_i}.
\end{equation}
Consequently,
\begin{equation}\label{exp-1}
\exp\Big\{-\frac{1-z(u;x,y)}{c_2t}\Big\}
\le  \exp\Big\{-\frac{d_\TT(x,y)^2}{c't}\Big\}
\prod_{i=1}^{d+1} \exp\Big\{-\frac{(1-u_i)\sqrt{x_iy_i}}{c't}\Big\}
\end{equation}
and
\begin{equation}\label{exp-2}
\exp\Big\{-\frac{1-z(u;x,y)}{c_4''t}\Big\}
\ge  \exp\Big\{-\frac{d_\TT(x,y)^2}{c''t}\Big\}
\prod_{i=1}^{d+1} \exp\Big\{-\frac{(1-u_i)\sqrt{x_iy_i}}{c''t}\Big\}.
\end{equation}

For $x,y\in [0,1]$ and $\kappa>0$, denote
\begin{equation}\label{A-T}
A_t(\kappa; x,y):= \kappa \int_{0}^1\exp\Big\{-\frac{(1-u)\sqrt{xy}}{ct}\Big\}(1-u)^{\kappa-1}du,
\end{equation}
where $c>0$ is a constant.
We claim that for any $0< \varepsilon \le 1$
\begin{equation}\label{est-At}
\frac{c_\diamond t^{|\kappa|}}{\prod_{i=1}^{d+1}(x_i+t)^{\kappa_i/2}(y_i+t)^{\kappa_i/2}}
\le \prod_{i=1}^{d+1} A_t(\kappa_i; x_i,y_i)
\le \frac{c^\diamond t^{|\kappa|}\exp\Big\{\varepsilon \frac{d_\TT(x,y)^2}{t}\Big\}}
{\prod_{i=1}^{d+1}(x_i+t)^{\kappa_i/2}(y_i+t)^{\kappa_i/2}},
\end{equation}
where $c^\diamond>0$ depends on $\varepsilon$.

Assume for a moment that the inequalities \eqref{est-At} are valid.
Then by \eqref{main-est3-TT}, \eqref{exp-1}, and the right-hand side inequality in \eqref{est-At} we obtain
\begin{align*}
e^{t\CD_\kappa}(x, y)
&\le  \frac{c}{t^{|\kappa|+d/2}}\exp\Big\{-\frac{d_\TT(x,y)^2}{c't}\Big\}
\frac{t^{|\kappa|}\exp\big\{\varepsilon\frac{d_\TT(x,y)^2}{t}\big\}}
{\prod_{i=1}^{d+1}(x_i+t)^{\kappa_i/2}\prod_{i=1}^{d+1}(y_i+t)^{\kappa_i/2}}
\\
&\le \frac{c\exp\big\{-\frac{d_\TT(x, y)^2}{2c't}\big\}}{\big[V_\TT(x, \sqrt{t})V_\TT(y, \sqrt{t})\big]^{1/2}}.
\end{align*}
Here we used that $\lambda=|\kappa|+(d-1)/2$, and $V_\TT(x, \sqrt{t})= t^{d/2}\prod_{i=1}^{d+1}(x_i+t)^{\kappa_i}$
and the similar expression for $V_\TT(y, \sqrt{t})$,
which follow by \eqref{V-ball-simpl}.
We also used the right-hand side estimate in \eqref{est-At} with $\varepsilon=(2c')^{-1}$.
The above inequalities yields the upper bound estimate in \eqref{gauss-simplex}.
One similarly shows that \eqref{main-est4-TT}, \eqref{exp-2}, and the left-hand side inequality in \eqref{est-At}
imply the lower bound estimate in \eqref{gauss-simplex}.

\smallskip

It remains to prove the estimates in \eqref{est-At}.
We first focus on the lower bound estimate in \eqref{est-At}.
If $\sqrt{xy}/t\le 1$, then
$\exp\big\{-\frac{(1-u)\sqrt{xy}}{ct}\big\} \ge c'>0$
and hence
\begin{equation*}
A_t(\kappa; x,y) \ge c' \ge c'(t/\sqrt{xy})^\kappa.
\end{equation*}
Assume $\sqrt{xy}/t>1$. Then applying the substitution $v=\frac{(1-u)\sqrt{xy}}{t}$
we obtain
\begin{align*}
A_t(\kappa; x,y)= \frac{t^\kappa}{(\sqrt{xy})^\kappa}\int_{0}^{\sqrt{xy}/t}e^{-v/c}v^{\kappa-1}dv
\ge \frac{t^\kappa}{(\sqrt{xy})^\kappa}\int_0^1e^{-v/c}v^{\kappa-1}dv
=\frac{c't^\kappa}{(\sqrt{xy})^\kappa}.
\end{align*}
Thus in both cases
\begin{equation*}
A_t(\kappa; x,y) \ge \frac{c't^\kappa}{(\sqrt{xy})^\kappa}
\ge \frac{c't^\kappa}{(x+t)^{\kappa/2}(y+t)^{\kappa/2}},
\end{equation*}
which yields the lower bound estimate in \eqref{est-At}.

\smallskip

We now prove the upper bound estimate in \eqref{est-At}.
Clearly
$\exp\big\{-\frac{(1-u)\sqrt{xy}}{ct}\big\}\le 1$ and hence
$A_t(\kappa; x,y)\le c'$.
On the other hand, from above it follows that
\begin{align*}
A_t(\kappa; x,y)
\le \frac{t^\kappa}{(\sqrt{xy})^\kappa}\int_0^\infty e^{-v/c}v^{\kappa-1}dv
=\frac{c''t^\kappa}{(\sqrt{xy})^\kappa}.
\end{align*}
Together, these two estimates yield
\begin{equation*}
A_t(\kappa; x,y) \le \frac{c^\star t^\kappa}{(\sqrt{xy}+t)^\kappa},
\end{equation*}
implying
\begin{equation}\label{est-prod-A}
\prod_{i=1}^{d+1} A_t(\kappa_i; x_i,y_i)
\le \frac{c^\star t^{|\kappa|}}{\prod_{i=1}^{d+1}\big(\sqrt{x_iy_i }+ t\big)^{\kappa_i}}.
\end{equation}
To show that this leads to the desired upper bound estimate,
we need the following simple inequality (see \cite[(2.50)]{IPX})
$$
\left|\sqrt{x_i}-\sqrt{y_i}\right|\le d_\TT(x,y),
\quad i=1, \dots, d+1, \;\; x,y \in \TT^d.
$$
This along with \eqref{elementary} implies
$$
\big(\sqrt{x_i}+ \sqrt{t}\big)\big(\sqrt{y_i}+\sqrt{t}\big)
\le c\big(\sqrt{x_iy_i}+ t\big)\Big(1+ \frac{d_\TT(x,y)}{\sqrt{t}}\Big),
$$
which leads to
\begin{align*}
\prod_{i=1}^{d+1}(x_i+t)^{\kappa_i/2}(y_i+t)^{\kappa_i/2}
& \sim \prod_{i=1}^{d+1}\big(\sqrt{t} + \sqrt{x_i}\big)^{\kappa_i}\big(\sqrt{t} + \sqrt{y_i}\big)^{\kappa_i}
\\
& \le c \prod_{i=1}^{d+1}\big(\sqrt{x_iy_i }+ t\big)^{\kappa_i}
\Big(1+ \frac{d_\TT(x,y)}{\sqrt{t}}\Big)^{|\kappa|}
\\
& \le c(\varepsilon) \prod_{i=1}^{d+1}\big(\sqrt{x_iy_i }+ t\big)^{\kappa_i}
\exp\Big\{\varepsilon|\kappa|\frac{d_\TT(x,y)^2}{t}\Big\}.
\end{align*}
Here for the last inequality we used \eqref{ineq-3} with $\mu=|\kappa|$.
Together, the above and \eqref{est-prod-A} yield the upper bound estimate in \eqref{est-At}.

\smallskip

We now consider the case when one or more $\kappa_i=0$, $1\le i\le n+1$. In this case,
the kernel representation \eqref{rep-Pn-simpl} holds under the limit. If $\kappa_i =0$, then the integral
over $u_i$ in \eqref{h-kernel-simplex2} is replaced by the average of point evaluations at $u_i=1$ and $u_i=-1$.
It is easy to see that all deductions that lead to \eqref{est-At} are still valid with the realization that \eqref{A-T}
holds under the limit
$$
\lim_{\kappa \to 0+} A_t(\kappa; x,y) =
\lim_{\kappa \to 0+} \kappa \int_{0}^1\exp\Big\{-\frac{(1-u)\sqrt{xy}}{ct}\Big\}(1-u)^{\kappa-1}du =1.
$$
This completes the proof.
\end{proof}

\end{document}